\documentclass[12pt,a4 paper]{article}
\usepackage[T1,T2A]{fontenc}
\usepackage[english]{babel}
\usepackage[utf8]{inputenc}
\usepackage{indentfirst}
\usepackage[left=30mm, top=20mm, right=25mm, bottom=25mm, nohead]{geometry}
\usepackage{csquotes}
\usepackage[backend=biber,bibencoding=utf8, style=ieee]{biblatex}
\usepackage{amsfonts}
\usepackage{amsmath}
\usepackage{amssymb}
\usepackage{amsthm}

\usepackage{mathrsfs}
\usepackage{xcolor}
\usepackage{hyperref}
\newtheorem{thm}{Theorem}[section]
\newtheorem{lmm}[thm]{Lemma}
\newtheorem{prp}[thm]{Proposition}
\newtheorem{crl}[thm]{Corollary}
\theoremstyle{definition}
\newtheorem{dfn}[thm]{Definition}
\newtheorem{ass}[thm]{Assumption}
\theoremstyle{remark}
\newtheorem{rmk}[thm]{Remark}
\newcommand{\dfnem}[1]{\textit{#1}}

\DeclareMathOperator{\Id}{Id}

\newcommand{\pp}{\mathscr{P}}
\renewcommand{\ll}{\mathscr{L}}
\newcommand{\xx}{\mathscr{X}}

\newcommand{\cc}{C^\infty_c}

\newcommand{\rr}{\mathbb{R}}
\newcommand{\rrp}{\rr_+}
\newcommand{\rrpo}{\rr_+^o}
\newcommand{\rd}{{\rr^d}}
\newcommand{\rds}{{\rr^{d*}}}

\newcommand{\hm}{{\hat{m}}}
\newcommand{\hk}{{\hat{k}}}
\newcommand{\hd}{{\hat{\delta}}}
\newcommand{\hatt}{{\hat{T}}}
\newcommand{\hpi}{{\hat{\pi}}}
\newcommand{\vk}{\varkappa}
\newcommand{\phivk}{\phi_\vk}
\newcommand{\muvk}{\mu_\vk^\varepsilon(m)}
\newcommand{\pivk}{\pi_\vk^\varepsilon(m)}
\newcommand{\pivki}{\pi_\vk^\varepsilon(m_{t_i})}
\newcommand{\gvk}{\gamma_\vk^\varepsilon(m)}
\newcommand{\yvk}{y_\vk^\varepsilon}
\newcommand{\kke}{k_\vk^\varepsilon}
\newcommand{\kkea}[2]{k_{#1}^{#2}}

\newcommand{\traj}[3]{\mathsf{X}^{{#1},{#2}}_{m_\cdot}(#3)}
\newcommand{\trajm}[3]{m_{#1}\left[m_*,#2,#3\right]}
\newcommand{\trajset}{\mathsf{\Theta}}

\newcommand{\tran}{\top}
\newcommand{\proj}{\mathfrak{p}}
\newcommand{\psd}{\partial_P^\varepsilon}
\renewcommand{\div}{\operatorname{div}}
\newcommand{\ball}{\mathsf{B}}

\newcommand{\oall}{\mathsf{O}}
\newcommand{\ord}{\mathsf{\Lambda}}
\newcommand{\ucf}{\mathsf{UC}_{\hm}}

\newcommand{\supin}{\mathcal{S}}
\newcommand{\infout}{\mathcal{I}}
\newcommand{\maxrad}{\mathcal{R}}
\newcommand{\modcon}{\omega_R^\varepsilon}
\newcommand{\bbrM}{\mathcal{M}^\varepsilon(R)}
\newcommand{\bbrMa}[1]{\mathcal{M}^{#1}(R)}
\newcommand{\bbr}{\mathcal{M}_\vk^\varepsilon(R)}
\newcommand{\bbra}[2]{\mathcal{M}_{#1}^{#2}(R)}
\newcommand{\bbrc}{\mathcal{K}_\vk^\varepsilon(R)}
\newcommand{\bbrh}{\mathcal{N}_\vk^\varepsilon(R)}
\newcommand{\bbrha}[2]{\mathcal{N}_{#1}^{#2}(R)}
\newcommand{\dlt}{\Delta}

\newcommand{\ds}{\displaystyle}

\title{Stabilization of solutions of the controlled non-local continuity equation}
\author{Volkov~A.}
\begin{document}
\maketitle
\section{Introduction}
The paper addresses the following problem: given a controlled non-local continuity equation
\[
\partial_t m_t + \div \big(f(x,m_t,u) m_t\big) = 0
\]
with an arbitrary initial distribution $m_*$, we aim to find a feedback control law that stabilizes the system to a given state.
The non-local continuity equation describes the evolution of an infinite system of identical particles, where the function $f$ represents a vector field that governs the motion of each particle based on its current position $x$ and the current spatial distribution $m$.
Note that the state variable is $m$, which is a probability measure.

The continuity equation is used to model various phenomena, including systems of charged particles~\cite{Boffi_1984}, the behavior of supermassive black holes~\cite{Tucci_2017} and large groups of animals~\cite{Botsford_1981}, dynamics of biological processes~\cite{Di_Mario_1993} and public opinion~\cite{Pluchino_2005}.

The problem under consideration traces back to similar issues in finite-dimensional cases.
Most results have been obtained for linear systems with continuous (or even smooth) controls~\cite{Coron_1992,Ikeda_1972,Sontag_1993,Arnold_1983,Kwon_1980}.
Attempts to generalize these results have often encountered counterexamples and the need for stronger assumptions~\cite{Brocket_1983,Coron_1995,Sontag_1980,Terrell_2009}.
A breakthrough in this area was achieved in~\cite{Clarke_1997,Clarke_1998,Sontag_1993}, where discontinuous synthesis was introduced based on the extremal shift rule, originally proposed for zero-sum differential games by Krasovskii and Subbotin in~\cite{Krasovskii_1974}.
The method developed in~\cite{Clarke_1998} applies to a broad class of nonlinear systems.
It relies on the concept of a control-Lyapunov pair (CLP), which consists of a Lyapunov function and an upper bound on its derivative.
Note that the construction of the stabilizing synthesis employs inf-convolution and techniques from proximal calculus.
It is shown that the existence of a CLP is sufficient to construct a stabilizing synthesis, and the opposite also holds.

In this paper, we extend the results of~\cite{Clarke_1998} to the case of a non-local continuity equation.
However, this extension is not straightforward because the state space now is neither linear nor $\sigma$-compact.
To overcome this, we employ variational principles and concepts from non-smooth analysis in the Wasserstein space~\cite{Ekeland_1974,Clarke_1998,Cardaliaguet_2019,Ambrosio_2005}.
We redefine the notion of a control-Lyapunov pair and prove that its existence is sufficient for both local and global stabilization.

The rest of the paper is organized as follows.

Section~\ref{sec:term} provides general definitions and notations used throughout the paper.
Here, we define the solution of the non-local continuity equation and discuss some of its properties.

Section~\ref{sec:inf} presents the first component of the approach developed in~\cite{Clarke_1998}, namely inf-convolution.
Due to the infinite-dimensional nature of the state space, we cannot guarantee the existence of a minimizer.
Instead, we use Ekeland's variational principle to get a pseudo-minimizer.

In Section~\ref{sec:sub}, we introduce a proximal subgradient for functions defined on the Wasserstein space, which serves as the second component of our approach.
Additionally, we show that every pseudo-minimizer of the inf-convolution is a barycenter of the proximal subgradient.

Local stabilization is discussed in Section~\ref{sec:local}.
Here, we introduce a feedback synthesis based on inf-convolution and prove that this synthesis achieves local stabilization.
Our proof relies heavily on the proximal subgradient and the control-Lyapunov pair.

Finally, Section~\ref{sec:global} addresses global stabilization.
We define the concept and construct a feedback law that achieves it.
The construction involves the sequential implementation of local stabilizing feedback laws.
\section{Preliminaries} \label{sec:term}
\subsection{General notions}

We use the following notation:

\begin{itemize}
\item $\rrp$ denotes the closed half-line $[0,+\infty)$;
\item $\rrpo$ denotes the open half-line $(0,+\infty)$;
\item $\rd$ is the space of $d$-dimensional column-vectors;
\item $\rds$ is the space of $d$-dimensional row-vectors;
\item ${}^\tran$ denotes the transpose operation;
\item $\ord^2$ is the set of ordered pairs of positive numbers, i.e.
\[
\ord^2 \triangleq \big\{ (x_1,x_2) \in \rr^2\ :\ 0 < x_1 < x_2 \big\}.
\]

\item If $(\Omega,\mathscr{F})$ and $(\Omega',\mathscr{F}')$ are measurable spaces, $m$ is a measure on the $\sigma$-algebra $\mathscr{F}$, and $g: \Omega \to \Omega'$ is $\mathscr{F}/\mathscr{F}'$-measurable, then $g\sharp m$ denotes the push-forward measure of $m$ through $g$, defined by:
\[
(g\sharp m)(\Upsilon)
= m\left( g^{-1}(\Upsilon) \right)
\]
for every $\Upsilon \in \mathscr{F}'$.

In particular, for any measurable functions $f: \Omega' \to \rr$ and $g: \Omega \to \Omega'$, the following equality holds:
\[
\int_{\Omega'} f(y)\ (g\sharp m)(dy)
= \int_\Omega f\big(g(x)\big)\ m(dx)
\]
\item $\ds \pp(\xx)$ stands for the space of Borel probability measures over $\xx$.
\item $\varsigma_2(\mu)$ denotes the square root of the second moment of the measure $\mu$, i.e.
\[
\varsigma_2(\mu) \triangleq \left( \int_\rd \rho(x,x_0)^2\ \mu(dx) \right)^{1/2} = \|\rho(\,\cdot\,,x_0)\|_{\ll_2(\mu)}, \text{ for arbitrary fixed } x_0.
\]
\item $\ds \pp_2(\xx)$ is the subset of $\ds \pp(\xx)$ consisting of measures with finite second moments.

\item $\ds\ball_R(x)$ denotes the closed ball of radius $R$ centered at $x$ in a metric space $(\xx,\rho)$, i.e.
\[
\ball_R(x) \triangleq \{ y \in \xx:\ \rho(x,y) \leq R \}.
\]
\item $\ds\oall_R(x)$ denotes the open ball of radius $R$ centered at $x$ in $(\xx,\rho)$, i.e.
\[
\oall_R(x) \triangleq \{ y \in \xx:\ \rho(x,y) < R \}.
\]
\item $\Id: \xx \to \xx$ is the identity operator defined by $\ds \Id(x) = x$ for all $x \in \xx$.
\item For $x =(x_i)_{i=1}^n \in \xx \times \dots \times \xx$ and an arbitrary multi-index $I$, $\proj^I$ denotes the projection operator:
\[
\proj^I(x) = (x_i)_{i \in I}.
\]
\end{itemize}
\medskip

For $\xx = \rd$, we additionally define:
\begin{itemize}
\item $\ll_2(\mu)$ denotes the space of Borel functions $\phi: \rd \to \rd$ satisfying:
\[
\int_\rd \| \phi(x) \|^2\ \mu(dx) < +\infty,
\]
with norm:
\[
\|\phi\|_{\ll_2(\mu)} \triangleq \left( \int_\rd \|\phi(x)\|^2\ \mu(dx) \right)^{1/2}.
\]
\item $\cc(\rd)$ stands for the space of infinitely differentiable functionals in $\rd$ with compact support.
\end{itemize}

\begin{dfn}[{\cite[Problem 1.2 and Section 5.1]{Santambrogio_2015}}]
A \dfnem{plan} between measures $\mu \in \pp(\rd)$ and $\nu\in \pp(\rd)$ is a measure $\pi \in \pp(\rd\times\rd)$ such that, for all Borel sets $A, B \subseteq \rd$,
\[
\pi(A \times \rd) = \mu(A), \quad \pi(\rd \times B) = \nu(B).
\]
The set of all plans is denoted by $\Pi(\mu,\nu)$.
\end{dfn}

\begin{dfn}[{\cite[Problem 1.2 and Section 5.1]{Santambrogio_2015}}] \label{dfn:w}
The \dfnem{Wasserstein distance} between measures $\mu\in\pp_2(\rd)$ and $\nu\in\pp_2(\rd)$ is defined by the rule:
\[
W_2(\mu,\nu) \triangleq \left(\inf\limits_{\pi\in\Pi(\mu,\nu)} \int\limits_{\rd\times\rd} \|x-y\|^2\ \pi(dxdy) \right)^{1/2}.
\]
A plan $\pi$ providing the infimum is called \dfnem{optimal}; the set of optimal plans is denoted by $\Pi_o(\mu,\nu)$.
\end{dfn}

\begin{rmk} \label{rmk.plan.wp}
Key properties of the Wasserstein distance are proved in~\cite{Santambrogio_2015}.
\begin{enumerate}
\item Set $\Pi_o$ is always non-empty (\cite[Problem~1.2 and Theorem~1.7] {Santambrogio_2015}).
\item The Wasserstein distance is a metric on the space $\pp_2(\rd)$(\cite[Proposition~5.1]{Santambrogio_2015}).
\item The metric space $\big(\pp_2(\rd), W_2\big)$ is Polish (\cite[Theorem~5.11]{Santambrogio_2015}).
\end{enumerate}
\end{rmk}

\begin{dfn}[{\cite[Definition 10.4.2]{Bogachev_2007}}]
Let $\xx_n = (\rd)^{\times n}$.
For measure $\pi \in \pp(\xx_n)$ and natural number $i < n$, a family of measures $\{ \hpi^i(\cdot | x_i) \}_{x_i \in \rd} \subseteq \pp(\xx_{n-1})$ is a \dfnem{disintegration of $\pi$ by the $i$-th variable} if for every bounded continuous $\phi: \xx_n \to \rr$:
\begin{multline*}
\int\limits_{\xx_n} \phi(x_1,\dots,x_n)\ \pi(dx_1 \dots dx_n)\\
= \int\limits_{\rd} \left( \int_{\xx_{n-1}} \phi(x_1,\dots,x_n)\ \hpi^i(dx_1 \dots dx_{i-1} dx_{i+1} \dots dx_n | x_i) \right)\ (\proj^i\sharp\pi)(dx_i),
\end{multline*}
\end{dfn}

\begin{rmk}
It follows from~\cite[Theorem~10.4.10]{Bogachev_2007}, that disintegrations exist for every measure of $\pp(\rd\times\dots\times\rd)$ by each variable.
\end{rmk}

\begin{dfn}
$\ucf$ is the class of functions $\phi: \pp_2(\rd) \to \rrpo$ those are uniformly continuous on every bounded set, and positive-definite at the point $\hm$, i.e., satisfying $\phi(\hm) = 0$ and $\phi(m) > 0$ for all $m \neq \hm$.
\end{dfn}
\subsection{Non-local continuity equation}
The main object of this paper is the initial value problem for the controlled non-local continuity equation
\begin{equation}\label{eq:sys}
\partial_t m_t + \div \big(f(x,m_t,u) m_t\big) = 0, \quad m_0 = m_*,
\end{equation}
where $m_* \in \pp_2(\rd)$, $U$ is a sequential compact and $f: \rd \times \pp_2(\rd) \times U \to \rd$ is a vector field.
As it was mentioned above, this equation describes an evolution of infinite particle system where each particle follows:
\begin{equation}\label{eq:sys.x}
\dot{x} = f(x,m_t,u).
\end{equation}

We assume the following condition:

\begin{ass} \label{ass.lip}
The function $f$ is continuous in $u$ and Lipschitz continuous in $x$ and $m$: there exists a constant $C_0 > 0$ such that, for every $x,y \in \rd$, $\mu,\nu \in \pp_2(\rd)$, and $u \in U$,
\[
\|f(x,\mu,u) - f(y,\nu,u)\|
\leq C_0 \big(\|x - y\| + W_2(\mu,\nu)\big).
\]
\end{ass}

\begin{rmk} \label{ass.ulin}
Setting $\nu=\delta_0$ and $y=0$ in Assumption~\ref{ass.lip} implies that function $f$ has sublinear growth: there exists a constant $C_1 > 0$, such that, for every $x \in \rd$, $\mu \in \pp_2(\rd)$ and $u \in U$,
\[
\|f(x,\mu,u)\|
\leq C_1\big(1 + \|x\| + \varsigma_2(\mu)\big).
\]
\end{rmk}

Now, we define feedback strategies and corresponding solutions to~\eqref{eq:sys}:

\begin{dfn} \mbox{}
\begin{itemize}
\item A \dfnem{feedback} is a function $k: \pp_2(\rd) \to U$.
\item A \dfnem{partition} of the half-line $[0,+\infty)$ is an infinite sequence $\theta = \{ t_i \}_{i=0}^{\infty}$ with $t_0 = 0$, $t_i < t_j$ whenever $i < j$, and $\lim\limits_{i\to\infty} t_i = \infty$.
\item $\trajset(\delta_1,\delta_2)$ denotes the set of partitions whose interval lengths are from $[\delta_1,\delta_2]$, i.e.
\[
\trajset(\delta_1,\delta_2) \triangleq \big\{\ \{ t_i \}_{i=0}^{\infty}\ :\ \delta_1 \leq t_{i+1} - t_i \leq \delta_2 \big\}.
\]
\item A measure-valued function $m_\cdot: \rrp \to \pp_2(\rd)$ is a \dfnem{solution} of continuity equation~\eqref{eq:sys} on the interval $[t_i,t_{i+1}]$ generated by fixed control $u \in U$ if for all $\varphi \in \cc\big(\rd \times (t_i,t_{i+1});\rr\big)$:
\[
\int_{t_i}^{t_{i+1}} \int_{\rd} \big( \partial_t \varphi(x,t) + \nabla_x \varphi(x,t) \cdot f(x,m_t,u) \big)\ m_t(dx)dt = 0.
\]
\item For a partition $\theta$, \dfnem{$\theta$-trajectory} of continuity equation~\eqref{eq:sys} under feedback $k$ is constructed by sequentially solving the initial value problems
\begin{gather*}
\partial_t m_t^i + \div \Big(f\big(x,m_t^i,k(m_{t_i}^i)\big) m_t^i\Big) = 0, \quad t \in [t_i,t_{i+1}], \\
m_0^0 = m_*, \quad m_{t_i}^i = m_{t_i}^{i-1}
\end{gather*}
and defining
\[
m_t = m_t^i, \quad \text{for } t \in [t_i,t_{i+1}],
\]
This trajectory at the time $t$ is denoted by $\trajm{t}{\theta}{k}$.

\item For a fixed $\theta$-trajectory $m_\cdot$ the solution of the initial value problem for a single particle~\eqref{eq:sys.x} with initial condition $x(s) = z$ evaluated at the time $r$ is denoted by $\ds \traj{s}{z}{r}$.
\end{itemize}
\end{dfn}

\begin{rmk}
Existence and uniqueness of solution on the finite interval are proved in~\cite{Averboukh_2022} under more general assumptions.
\end{rmk}
\medskip

Finally, we provide technical results for some fixed $\hm \in \pp_2(\rd)$ and $R \geq 0$ those was proved in~\cite[Proposition~5.2 and Proposition~2.16]{Averboukh_2025} (see also~\cite[Proposition~A.2 and Proposition~2.16]{Averboukh_2024_arxiv} for english version).

\begin{prp} \label{prp:prev24}
If $t \in [t_i,t_{i+1}]$, $m_\cdot = \trajm{\cdot}{\theta}{k}$, $m_{t_i} \in \ball_R(\hm)$ and $(t_{i+1} - t_i) < \delta$, then, for all $t \in [t_i,t_{i+1}]$,
\[
W_2(m_t,m_{t_i})
\leq C_2(R) \cdot (t - t_i),
\]
where
\[
C_2(R)
= C_1 \cdot e^{C_1 \delta} \cdot \left(1 + 2(C_1 \delta + \varsigma_2(\hm) + R) \cdot e^{2 C_1 \delta}\right).
\]
\end{prp}

\begin{prp} \label{prp:prev25}
Under the assumptions of Proposition~\ref{prp:prev24}, for all $t \in [t_i,t_{i+1}]$,
\[
\left( \int_\rd \| f\big(\traj{t_i}{x_i}{t},m_t,k(m_{t_i})\big) - f\big(x_i,m_{t_i},k(m_{t_i})\big) \|^2\ m_{t_i}(dx_i) \right)^{1/2}
\leq C_3(R) (t - t_i),
\]
where
\[
C_3(R)
= C_0 \cdot \left( C_1 \cdot e^{C_1 \delta} \cdot \left(1 + 2(C_1 \delta +\varsigma_2(\hm) + R) \cdot e^{2 C_1 \delta} \right) + C_2(R) \right).
\]
\end{prp}
\section{Inf-convolution} \label{sec:inf}
\begin{dfn}
Let $\phi$ be l.s.c., $\vk \in (0,1]$ and $m \in \pp_2(\rd)$.
The \dfnem{inf-convolution} of the function $\phi$ with parameter $\vk$ at the point $m$ is defined as:
\begin{equation} \label{eq:inf}
\phivk(m)
\triangleq \inf\limits_{\substack{\mu \in \pp_2(\rd) \\ \pi \in \Pi(m,\mu)}}
\left[
\phi(\mu) + \frac1{2\vk^2} \int_{(\rd)^2} \|y-x\|^2\ \pi(dxdy).
\right]
\end{equation}

This operation generalizes the seminal Yosida-Moreau regularization (see~\cite{Yosida_1948,Moreau_1965,Showalter_2013}) to Wasserstein spaces.
\end{dfn}

\begin{rmk}
It is enough to consider only optimal plans $\pi$, i.e.,
\[
\phivk(m)
= \inf\limits_{\mu \in \pp_2(\rd)}
\left[
\phi(\mu) + \frac1{2\vk^2} W_2^2(m,\mu)
\right].
\]
\end{rmk}

\begin{rmk} \label{rmk:phivk.less.phi}
Considering the identical plan we get that $\phivk(m) \leq \phi(m)$ for every measure $m \in \pp_2(\rd)$ and parameter $\vk \in (0,1]$.
\end{rmk}

\begin{rmk} \label{rmk:ekeland}
Since the Wasserstein space is not finite-dimensional, we cannot ensure the existence of a minimizer.
In that case we use the Ekeland's variational principle introduced in~\cite{Ekeland_1974}: for all $m \in \pp_2(\rd)$, $\vk \in (0,1]$, and $\varepsilon > 0$, there exists a measure $\muvk \in \pp_2(\rd)$ such that, for all $\mu \in \pp_2(\rd)$,
\begin{enumerate}
\item $\ds \phi\big(\muvk\big) + \frac1{2\vk^2} W_2^2\big(m,\muvk\big) \leq \phi(m) - \varepsilon \cdot W_2(\muvk,m)$,
\item $\ds \phi\big(\muvk\big) + \frac1{2\vk^2} W_2^2\big(m,\muvk\big) \leq \phi(\mu) + \frac1{2\vk^2} W_2^2(m,\mu) + \varepsilon \cdot W_2(\muvk,\mu)$.
\end{enumerate}
\end{rmk}

Further, we need some additional notions using a target measure $\hm$.

\begin{dfn} \label{dfn:s.i.g.r}
Given measure $\hm \in \pp_2(\rd)$, numbers $R > 0$, $\vk \in (0,1]$ and function $\phi \in \ucf$, define the functions
\[
\supin(R) \triangleq \sup\limits_{\mu \in \ball_R(\hm)} \phi(\mu), \quad
\infout(R) \triangleq \inf\limits_{\mu \notin \oall_R(\hm)} \phi(\mu),
\]
the sets
\[
G_R \triangleq \left\{ \mu\ \Big|\ \phi(\mu) \leq \frac12 \infout(R) \right\}, \quad
G_R^\vk \triangleq \left\{ \mu\ \Big|\ \phivk(\mu) \leq \frac12 \infout(R) \right\},
\]
and the function $\maxrad(R) \triangleq \sup\{ r \geq 0\ |\ \ball_r(\hm) \subseteq G_R \}$.
\end{dfn}

The following relations are easy to prove.
\begin{prp}
Let $\hm \in \pp_2(\rd)$, $\phi \in \ucf$, $R > 0$, and $\vk \in (0,1]$.
Then:
\begin{enumerate}
\item $\infout(R) \leq \supin(R)$;
\item $\ds \maxrad(R) > 0$;
\item $\ds \lim_{R \to 0} \supin(R) = \lim \maxrad(R) = 0$;
\item $\ds \lim_{R \to \infty} \infout(R) = \lim \maxrad(R) = \infty$;
\item $\ds \oall_{\maxrad(R)}(\hm) \subseteq G_R \subseteq G_R^\vk$.
\end{enumerate}
\end{prp}

For fixed $\hm \in \pp_2(\rd)$, $\phi \in \ucf$, $R > 0$, $\vk \in (0,1]$, and $\varepsilon>0$  we define:
\begin{itemize}
\item $\ds \bbr \triangleq \vk \sqrt{2\supin(R) + \varepsilon^2 \vk^2} - \varepsilon\vk^2$;
\item $\ds \bbrM \triangleq R + \bbra{1}{\varepsilon} + \varepsilon\vk^2 = R + \sqrt{2\supin(R) + \varepsilon^2}$;
\item $\ds \bbrh \triangleq \varepsilon\sqrt{2\vk^2 \cdot \phivk(m) } + \varepsilon \cdot \bbr$;
\item $\ds \modcon(\delta) \triangleq \sup \Big\{ |\phi(\nu_1) - \phi(\nu_2)| \ \big|\ W_2(\nu_1,\nu_2) \leq \delta, \ \nu_1,\nu_2 \in \ball_{\bbrM}(\hm) \Big\}$;
\item $\ds \bbrc \triangleq \varepsilon\vk^2 + \vk \sqrt{\varepsilon^2 \vk^2 + 2\modcon(\bbr)}$.
\end{itemize}

Below we provide some estimates for the measure $\muvk$.

\begin{lmm} \label{lmm:w2.m.muvk} 
For $\vk \in (0,1]$, $\varepsilon > 0$, and $m \in \ball_R(\hm)$, the following inequality holds:
\[
W_2\big(m,\muvk\big)
\leq \bbr.
\]
\end{lmm}

\begin{proof}
By the first inequality of Remark~\ref{rmk:ekeland}:
\[
\phi\big(\muvk\big) + \frac1{2\vk^2} W_2^2\big(m,\muvk\big)
\leq \phi(m) - \varepsilon W_2\big(m,\muvk\big).
\]
Thus,
\begin{equation} \label{eq:lmm.w2.m.muvk}
\frac1{2\vk^2} W_2^2\big(m,\muvk\big) + \varepsilon W_2\big(m,\muvk\big)
\leq \phi(m) - \phi\big(\muvk\big)
\leq \phi(m)
\leq \supin(R).
\end{equation}
Solving this inequality yields
\[
W_2\big(m,\muvk\big)
\leq \vk \sqrt{2\supin(R) + \varepsilon^2 \vk^2} - \varepsilon\vk^2.
\]
\end{proof}

Using inequality~\eqref{eq:lmm.w2.m.muvk} and the fact that $\bbr \leq \bbrM$, it is easy to prove the corollarial inequalities.

\begin{crl} \label{crl:w2.muvk.m.hm}
In assumptions of Lemma~\ref{lmm:w2.m.muvk}:
\begin{enumerate}
\item $\ds W_2\big(m,\muvk\big) \leq \bbrc$,
\item $\ds W_2\big(\hm,\muvk\big) \leq \bbrM$.
\end{enumerate}
\end{crl}

Now, we prove inequalities for function $\phi$ and its inf-convolution.

\begin{lmm} \label{lmm:min.seq}
For $\vk \in (0,1]$, $\varepsilon > 0$, $m \in \ball_R(\hm)$, and $\phi \in \ucf$, the following inequality holds true:
\[
\phi\big(\muvk\big) + \frac1{2\vk^2} W_2^2\big(m,\muvk\big)
\leq \phivk(m) + \bbrh.
\]
\end{lmm}
\begin{proof}
Due to the very definition of the inf-convolution, there exists a minimizing sequence $\{\mu_n\}_{n=1}^\infty$, i.e.
\[
\phi(\mu_n) + \frac1{2\vk^2} W_2^2(m,\mu_n) \to \phivk(m) \quad \text{ as } n \to \infty.
\]
By the second inequality of Remark~\ref{rmk:ekeland} and the triangle inequality, we have that
\begin{align*}
&\phi\big(\muvk\big) + \frac1{2\vk^2} W_2^2\big(m,\muvk\big) \\
&\hspace{2cm}\leq \phi(\mu_n) + \frac1{2\vk^2} W_2^2(m,\mu_n) + \varepsilon \cdot W_2(\muvk,\mu_n) \\
&\hspace{2cm}\leq \phi(\mu_n) + \frac1{2\vk^2} W_2^2(m,\mu_n) + \varepsilon \cdot W_2(m,\mu_n) + \varepsilon \cdot W_2\big(m,\muvk\big).
\end{align*}
Notice, that
\[
\varepsilon W_2(m,\mu_n)
= \varepsilon\sqrt{2\vk^2} \cdot \sqrt{\frac1{2\vk^2} \cdot W_2^2(m,\mu_n)}
\leq \varepsilon\sqrt{2\vk^2} \cdot \sqrt{ \phi(\mu_n) + \frac1{2\vk^2} \cdot W_2^2(m,\mu_n)}
\]
Then, using both previous inequalities and Lemma~\ref{lmm:w2.m.muvk}, we arrive at the estimate
\begin{multline*}
\phi\big(\muvk\big) + \frac1{2\vk^2} W_2^2\big(m,\muvk\big) \\
\leq \phi(\mu_n) + \frac1{2\vk^2} W_2^2(m,\mu_n) + \varepsilon\sqrt{2\vk^2} \cdot \sqrt{ \phi(\mu_n) + \frac1{2\vk^2} \cdot W_2^2(m,\mu_n)} + \varepsilon \cdot \bbr.
\end{multline*}
Passing to the limit when $n \to \infty$, we obtain
\[
\phi\big(\muvk\big) + \frac1{2\vk^2} W_2^2\big(m,\muvk\big)
\leq \phivk(m) + \bbrh.
\]
\end{proof}

\begin{lmm} \label{lmm:phi.lesser.phivk} 
In assumptions of Lemma~\ref{lmm:min.seq}:
\[
\phi(m) - \phivk(m)
\leq \modcon\left(\bbr\right) + \bbrh.
\]
\end{lmm}

\begin{proof}
Due to Lemma~\ref{lmm:min.seq}, we have that
\[
\phi\big(\muvk\big)
\leq \phi\big(\muvk\big) + \frac1{2\vk^2} W_2^2\big(m,\muvk\big)
\leq \phivk(m) + \bbrh
\]
Using this inequality and Lemma~\ref{lmm:w2.m.muvk}, we obtain
\[
\phi(m) - \phivk(m)
\leq \phi(m) - \phi\big(\muvk\big) + \bbrh
\leq \modcon\left(\bbr\right) + \bbrh.
\]
\end{proof}

We complete this section with the fact that the set $G_R^\vk$ is bounded by the radius $R$.

\begin{lmm} \label{lmm:grvk.in.oall} 
For every $\vk>0$, $\varepsilon>0$, and $R > 0$ satisfying
\begin{equation} \label{eq:lmm.grvk.in.oall}
\modcon\left(\bbr\right) + \bbrh
< \frac12 \infout(R),
\end{equation}
the following inclusion holds:
\[
G_R^\vk
\subseteq \oall_R(\hm).
\]
\end{lmm}

\begin{proof}
For $m \in G_R^\vk$, by the definition of $G_R^\vk$, we have
\[
\phivk(m)
\leq \frac12 \infout(R).
\]
This inequality, Lemma~\ref{lmm:phi.lesser.phivk} and the condition of the lemma give that
\[
\phi(m)
\leq \phivk(m) + \modcon\left(\bbr\right) + \bbrh
< \infout(R).
\]
Hence, $m \in \oall_R(\hm)$.
\end{proof}
\section{Proximal subgradient} \label{sec:sub}
As mentioned earlier, the Wasserstein space lacks linear structure, making the definition of a direction non-trivial. Following~\cite{Ambrosio_2005}, we represent directions as distributions on the tangent bundle, with subdifferential elements residing in the space of probability measures over the cotangent bundle to $\rd$.

Let $\hm \in \pp_2(\rd)$ and $\varepsilon > 0$.

\begin{dfn}
A measure $\alpha \in \pp_2(\rd \times \rds)$ is a \dfnem{proximal $\varepsilon$-subgradient} of the function $\phi: \pp_2(\rd) \to \rr$ at the point $m \in \pp_2(\rd)$ if there exist $\sigma>0$ and $r>0$ such that, for all $\mu \in \ball_r(m)$ and every $\beta \in \pp_2(\rd \times \rd \times \rds)$ satisfying
\begin{enumerate}
\item $\ds \proj^{1,2}\sharp\beta \in \Pi(m,\mu)$,
\item $\ds \proj^{1,3}\sharp\beta = \alpha$,
\end{enumerate}
the following inequality holds:
\begin{multline*}
\phi(\mu) \geq \phi(m)
+ \int_{(\rd)^3} p \cdot (x_2 - x_1)\ \beta(dx_1dx_2dp) \\
- \sigma \int_{(\rd)^2} \| x_2 - x_1 \|^2\ (\proj^{1,2}\sharp\beta)(dx_1dx_2)
- \varepsilon \cdot W_2(m,\mu).
\end{multline*}

The set of proximal $\varepsilon$-subgradients is called a \dfnem{proximal $\varepsilon$-subdifferential} and denoted by $\psd \phi(m)$.
\end{dfn}

Define a map $\chi_\vk:\ \rd \times \rd \to \rd \times \rds$ and a measure $\gvk \in \pp_2(\rd \times \rds)$ by the rules:
\begin{equation} \label{eq:gvk}
\chi_\vk(z_1,z_2) \triangleq \left(z_2,\ \left(\frac{z_1-z_2}{\vk^2}\right)^\tran \right);
\quad \gvk \triangleq \chi_\vk \sharp \pivk,
\end{equation}
where $\pivk \in \Pi_o\big(m,\muvk\big)$.

\begin{lmm} \label{lmm:gamma.subgrad} 
The measure $\gvk$ is a proximal $\varepsilon$-subgradient of the function $\phi$ at the point $\muvk$.
\end{lmm}

\begin{proof}
For an arbitrary measure $\mu \in \pp_2(\rd)$ and a measure $\beta \in \pp_2(\rd \times \rd \times \rds)$ with properties
\begin{enumerate}
\item $\ds \proj^{1,2}\sharp\beta \in \Pi(\muvk,\mu)$;
\item $\ds \proj^{1,3}\sharp\beta = \gvk$,
\end{enumerate}
we construct a map $\xi_\vk:\ \rd \times \rd \times \rds \to \rd \times \rd \times \rd$ and a measure $\rho \in \pp_2(\rd \times \rd \times \rd)$ by the rules:
\begin{align*}
\xi_\vk(z_1,z_2,z_3) \triangleq \left(z_1 + \vk^2 \cdot z_3^\tran,\ z_1,\ z_2 \right);
\quad \rho \triangleq \xi_\vk\sharp\beta.
\end{align*}
Notice that
\begin{itemize}
\item $\ds (\xi_\vk)^{-1}(w_1,w_2,w_3) = \left(w_2,\ w_3,\ \left(\frac{w_1-w_2}{\vk^2}\right)^\tran \right)$
\item $\ds \proj^{1,2}\sharp\rho = \pivk$;
\item $\ds \proj^{1,3}\sharp\rho \in \Pi(m,\mu)$.
\end{itemize}
Let $\pi \triangleq \proj^{1,3}\sharp\rho$.
Then, by the second inequality of Remark~\ref{rmk:ekeland},
\begin{align*}
\phi(\mu)
&\geq \phi\big(\muvk\big) + \frac1{2\vk^2} \int_{(\rd)^2} \|\yvk-x\|^2\ \big(\pivk\big)(dxd\yvk) \\
&\hspace{3.5cm}- \frac1{2\vk^2} \int_{(\rd)^2} \|y-x\|^2\ \pi(dxdy) - \varepsilon \cdot W_2(\muvk,\mu) \\
&= \phi\big(\muvk\big) + \frac1{2\vk^2} \int_{(\rd)^3} \|\yvk-x\|^2\ \rho(dx d\yvk dy) \\
&\hspace{3.5cm}- \frac1{2\vk^2} \int_{(\rd)^3} \|y-x\|^2\ \rho(dx d\yvk dy) - \varepsilon \cdot W_2(\muvk,\mu) \\
&= \phi\big(\muvk\big) + \int_{(\rd)^3} \left( \frac1{2\vk^2} \|\yvk-x\|^2
- \frac1{2\vk^2} \|y-x\|^2 \right)\ \rho(dx d\yvk dy) \\
&\hspace{10cm}- \varepsilon \cdot W_2(\muvk,\mu).
\end{align*}
Expanding and rearranging terms yields
\begin{multline*}
\phi(\mu)
\geq \phi\big(\muvk\big) + \int_{(\rd)^3} \left( \left(\frac{x-\yvk}{\vk^2}\right)^\tran \cdot (y-\yvk)
- \frac1{2\vk^2} \|y-\yvk\|^2 \right)\ \rho(dx d\yvk dy) \\
- \varepsilon \cdot W_2(\muvk,\mu).
\end{multline*}
Thus,
\begin{multline*}
\phi(\mu)
\geq \phi\big(\muvk\big) + \int_{(\rd)^3} p \cdot (y-\yvk)\ \beta(d\yvk dy dp) \\
- \frac1{2\vk^2} \int_{(\rd)^2} \| y - \yvk \|^2\ (\proj^{1,2}\sharp\beta)(d\yvk dy) - \varepsilon \cdot W_2(\muvk,\mu),
\end{multline*}
proving $\gvk \in \psd\phi\big(\muvk\big)$.
\end{proof}

The following statement can be regarded as an analog of the Taylor expantion.

\begin{lmm} \label{lmm:taylor} 
Let $m \in \ball_R(\hm)$, $b \in \ll_2(m)$, $\tau>0$ and $\mu = (\Id+\tau b)\sharp m$.
Then,
\begin{multline*}
\phivk(\mu)
\leq \phivk(m) + \int_{(\rd)^2} \left( \frac{x-\yvk}{\vk^2} \right)^\tran \cdot \tau b(x)\ \big(\pivk\big)(dxd\yvk) \\
+ \frac{\tau^2}{2\vk^2} \|b\|_{\ll_2(m)}^2
+ \bbrh
\end{multline*}
\end{lmm}

\begin{proof}
By the definition of $\phivk$ (see~\eqref{eq:inf}), we have that
\[
\phivk(\mu)
\leq \phi\big(\muvk\big)
+ \frac1{2\vk^2} \int_{(\rd)^2} \|\yvk-y\|^2\ \pi(dyd\yvk)
\]
for all $\mu \in \pp_2(\rd)$ and $\pi \in \Pi\big(\mu,\muvk\big)$.
Define a mapping $\xi: \rd \times \rd \to \rd \times \rd$ by the rule:
\[
\xi(z_1,z_2) \triangleq \big( z_1 + \tau b(z_1),\ z_2 \big),
\]
and choose a plan $\pi \triangleq \xi\sharp\pivk$.
Expanding the squared norm, we obtain
\begin{align*}
\phivk(\mu)
&\leq \phi\big(\muvk\big)
+ \frac1{2\vk^2} \int_{(\rd)^2} \|y-\yvk\|^2\ \pi(dyd\yvk) \\
&= \phi\big(\muvk\big)
+ \frac1{2\vk^2} \int_{(\rd)^2} \|x+\tau b(x)-\yvk\|^2\ \big(\pivk\big)(dxd\yvk) \\
&= \phi\big(\muvk\big)
+ \frac1{2\vk^2} \int_{(\rd)^2} \|x-\yvk\|^2\ \big(\pivk\big)(dxd\yvk) \\
&\hspace{2cm}+ \int_{(\rd)^2} \left(\frac{x-\yvk}{\vk^2}\right)^\tran \cdot \tau b(x)\ \big(\pivk\big)(dxd\yvk)
+ \frac{\tau^2}{2\vk^2} \|b\|_{\ll_2(m)}^2.
\end{align*}
Since $\pivk$ is an optimal plan between $m$ and $\muvk$, Lemma~\ref{lmm:min.seq} implies
\begin{multline} \label{eq:lmm.taylor}
\phivk(\mu)
\leq \phivk(m) + \bbrh \\
+ \int_{(\rd)^2} \left(\frac{x-\yvk}{\vk^2}\right)^\tran \cdot \tau b(x)\ \big(\pivk\big)(dxd\yvk)
+ \frac{\tau^2}{2\vk^2} \|b\|_{\ll_2(m)}^2
\end{multline}
\end{proof}
\medskip

We conclude this section with the notion of control-Lyapunov function extending the definition from~\cite{Clarke_1997} to the Wasserstein space.

\begin{dfn}[\cite{Ledyaev_2010}] \label{dfn:clp}
A \dfnem{control-Lyapunov pair (CLP)} at the point $\hm$ for continuity equation~\eqref{eq:sys} is a pair of two continuous functions $(\phi,\psi)$, where $\phi \in \ucf$ and $\psi: \pp_2(\rd) \times \rrpo \to \rrpo$ satisfy the following assumptions:
\begin{enumerate}
\item for every $R > 0$ the level-set $\ds \{ \mu \in \pp_2(\rd)\ :\ \phi(\mu) \leq R \}$ is bounded;
\item for all $(r,R) \in \ord^2$, there exists a function $g: \rrpo \to \rrpo$ such that
\[
\inf \big\{ \psi(m,\varepsilon)\ :\ m \in \ball_R(\hm)\setminus\oall_r(\hm) \big\} \geq g(\varepsilon) > 0;
\]
\item the mapping $\varepsilon \mapsto \psi(m,\varepsilon)$ is monotone increasing;
\item there exist $\varepsilon_0 > 0$ such that, for all $\varepsilon \in (0,\varepsilon_0)$, measure $m \in \pp_2(\rd)$ and subgradient $\alpha \in \psd\phi(m)$ the following inequality holds true:
\[
\min\limits_{u \in U} \int_{\rd\times\rd} p \cdot f(x,m,u)\ \alpha(dxdp) \leq -\psi(m,\varepsilon).
\]
\end{enumerate}
The function $\phi$ is called a \dfnem{control-Lyapunov function (CLF)}.
\end{dfn}
\section{Local stabilization at neighborhood} \label{sec:local}
In this section, we fix a measure $\hm \in \pp_2(\rd)$, and a pair of numbers $(r,R) \in \ord^2$.
Also we assume that the parameter $\varepsilon$ is from $(0,1]$.
Let $(\phi,\psi)$ be a CLP at point $\hm$ for continuity equation~\eqref{eq:sys}.
Additionally, for each $m \in \pp_2(\rd)$, let $\pivk \in \Pi_o\big(m,\muvk\big)$ and
\begin{equation} \label{eq:Delta.r.r}
\dlt(r,R)
\triangleq \frac13 \inf \left\{ \psi(\mu,0)\ \Big|\ \mu \in \pp_2(\rd);\  \frac12 \maxrad(r) \leq W_2(\hm,\mu) \leq \bbrMa{1} \right\} > 0.
\end{equation}

We define the feedback strategy $\kke(m)$ by the rule:
\begin{multline} \label{eq:feedback.local}
\int_\rd \left(\frac{x-\yvk}{\vk^2}\right)^\tran \cdot f\big(x,m,\kke(m)\big)\ \big(\pivk\big)(dxd\yvk) \\
= \min\limits_{u \in U} \int_\rd \left(\frac{x-\yvk}{\vk^2}\right)^\tran \cdot f(x,m,u)\ \big(\pivk\big)(dxd\yvk).
\end{multline}
Note, that this strategy implements the extremal shift at the direction determined by $\pivk$.
This concept goes back to the seminal book by Krasovskii and Subbotin~\cite{Krasovskii_1974}.

Now, we prove that the set $G_R^\vk$ is invariant w.r.t. $\theta$-trajectories generated by the feedback $\kke$ for sufficiently small $\varepsilon$.

\begin{lmm} \label{lmm:2Delta} 
For $\vk, \varepsilon \in (0,1]$ satisfying
\begin{equation} \label{eq:lmm.2Delta}
\bbr < \frac12 \maxrad(r), \quad \frac{2C_0}{\vk^2} \cdot \big( \bbrc \big)^2 < \dlt(r,R),
\end{equation}
and for all $m \in \ball_R(\hm) \setminus \ball_{\maxrad(r)}(\hm)$, the following inequality holds:
\[
\int_\rd \left( \frac{x-\yvk}{\vk^2}\right)^\tran \cdot f\big(x,m,\kke(m)\big)\ \big(\pivk\big)(dxd\yvk)
\leq -2\dlt(r,R).
\]
\end{lmm}

\begin{proof}
Due to the construction of the feedback strategy $\kke$ (see~\eqref{eq:feedback.local}), we have
\begin{equation} \label{eq:lmm.2Delta.lhs}
\begin{split}
\int_\rd &\left(\frac{x-\yvk}{\vk^2}\right)^\tran \cdot f\big(x,m,\kke(m)\big)\ \big(\pivk\big)(dxd\yvk) \\
&= \min\limits_{u \in U} \int_\rd \left(\frac{x-\yvk}{\vk^2}\right)^\tran \cdot f(x,m,u)\ \big(\pivk\big)(dxd\yvk) \\
&= \min\limits_{u \in U} \Bigg[ \int_\rd \left(\frac{x-\yvk}{\vk^2}\right)^\tran \cdot \big(f(x,m,u) - f(\yvk,\muvk,u)\big)\ \big(\pivk\big)(dxd\yvk) \\
&\hspace{4cm}+ \int_\rd \left(\frac{x-\yvk}{\vk^2}\right)^\tran \cdot f(\yvk,\muvk,u)\ \big(\pivk\big)(dxd\yvk) \Bigg]
\end{split}
\end{equation}
By Assumption~\ref{ass.lip}, the first term can be estimated by
\begin{align*}
&\int_\rd \frac{\|x-\yvk\|}{\vk^2} \cdot C_0 \cdot \big(\|x-\yvk\| + W_2\big(m,\muvk\big)\big)\ \big(\pivk\big)(dxd\yvk) \\
&\hspace{3cm}= \frac{C_0}{\vk^2} \int_\rd \|x-\yvk\|^2\ \big(\pivk\big)(dxd\yvk) \\
&\hspace{5cm}+ \frac{C_0 \cdot W_2\big(m,\muvk\big)}{\vk^2} \int_\rd \|x-\yvk\|\ \big(\pivk\big)(dxd\yvk)
\end{align*}
Using the H\"{o}lder's inequality and the first inequality at Corollary~\ref{crl:w2.muvk.m.hm}, we estimate the right-hand side of this formula by the value
\[
\frac{2C_0}{\vk^2} \cdot W_2^2\big(m,\muvk\big)
\leq \frac{2C_0}{\vk^2} \cdot \big( \bbrc \big)^2.
\]

The second term in~\eqref{eq:lmm.2Delta.lhs} can be rewritten as
\begin{multline*}
\int_\rd \left( \frac{x-\yvk}{\vk^2}\right)^\tran \cdot f(\yvk,\muvk,u) \ \big(\pivk\big)(dxd\yvk) \\
= \int_\rd p \cdot f(\yvk,\muvk,u)\ \big(\gvk\big)(d\yvk dp),
\end{multline*}
where $\gvk$ is defined by~\eqref{eq:gvk}.
Hence, by Lemma~\ref{lmm:gamma.subgrad} and Definition~\ref{dfn:clp}, we have that
\[
\min\limits_{u \in U} \int_\rd p \cdot f(\yvk,\muvk,u)\ \big(\gvk\big)(d\yvk dp)
\leq -\psi\big(\muvk,\varepsilon\big).
\]
So, the left-hand side of the~\eqref{eq:lmm.2Delta.lhs} satisfies
\begin{multline} \label{eq:lmm.2Delta.estim}
\int_\rd \left( \frac{x-\yvk}{\vk^2}\right)^\tran \cdot f\big(x,m,\kke(m)\big)\ \big(\pivk\big)(dxd\yvk)\\
\leq -\psi\big(\muvk,\varepsilon\big) + \frac{2C_0}{\vk^2} \cdot \bbrc^2.
\end{multline}

Further, by Lemma~\ref{lmm:w2.m.muvk},
\[
W_2(\hm,m) - W_2\big(\hm,\muvk\big)
\leq W_2\big(m,\muvk\big)
\leq \bbr.
\]
Using the first inequality in~\eqref{eq:lmm.2Delta}, we derive that
\[
W_2\big(\hm,\muvk\big)
\geq W_2(\hm,m) - \bbr
\geq \maxrad(r) - \bbr
\geq \frac12 \maxrad(r).
\]
Hence, the definition of $\dlt(r,R)$ (see~\eqref{eq:Delta.r.r}) gives the estimate
\[
\psi\big(\muvk,\varepsilon\big) \geq 3\dlt(r,R).
\]
This inequality, the conditions of the lemma and estimate~\eqref{eq:lmm.2Delta.estim} yields the desired inequality.
\end{proof}

Based on this result we can prove that a CLF is decreasing along the trajectory generated by the feedback $\kke$.

\begin{lmm} \label{lmm:phivk.notinc}
Let the following conditions hold true:
\begin{enumerate}
\item $\vk, \varepsilon \in (0,1]$ satisfy conditions~\eqref{eq:lmm.grvk.in.oall} and~\eqref{eq:lmm.2Delta};
\item $\delta_2 > 0$ satisfies
\begin{multline} \label{eq:lmm.phivk.notinc}
\dlt(r,R) > \delta_2 \cdot \bigg( \bbr \cdot C_3(R) \\
+ \frac{C_1^2}{2\vk^2} \cdot \big(1 + 4(\varsigma_2(\hm) + R) + 4(\varsigma_2(\hm) + R)^2\big) \bigg);
\end{multline}
\item partition $\theta = \{t_i\}_{i=1}^{\infty} \in \trajset(\delta_1,\delta_2)$ for some $\delta_1 \in (0,\delta_2)$;
\item $\trajm{t_i}{\theta}{\kke} \in G_R^\vk \setminus B_{\maxrad(r)}(\hm)$ for some natural $i$.
\end{enumerate}
Then, there exists $\varepsilon_1 \in (0,1]$ such that, for $\varepsilon \in (0,\varepsilon_1)$ and $t \in [t_i,t_{i+1}]$, the $\theta$-trajectory $m_\cdot = \trajm{\cdot}{\theta}{\kke}$ satisfies the following inequality:
\[
\phivk(m_t) - \phivk(m_{t_i})
< -\dlt(r,R) \cdot (t-t_i) + \bbrh
\]
and, consequently, $m_t \in G_R^\vk$.
\end{lmm}

\begin{proof}
We consider the maximal time $\hatt \in (t_i,t_{i+1}]$ such that $m_\tau \in \ball_R(\hm)$ for all $\tau \in [t_i,\hatt]$.
The existence of such $\hatt$ directly follows from the fact that $m_{t_i} \in O_R(\hm)$.

For $t \in [t_i,\hatt]$ and $x_i \in \rd$, we introduce the functions
\begin{align*}
g_i(t,x_i)
&\triangleq \frac1{t-t_i} \int_{t_i}^t f\big(\traj{t_i}{x_i}{\tau},m_\tau,\kke(m_{t_i})\big) - f\big(x_i,m_{t_i},\kke(m_{t_i})\big)\ d\tau; \\
f_i(t,x_i)
&\triangleq \frac1{t-t_i} \int_{t_i}^t f\big(\traj{t_i}{x_i}{\tau},m_\tau,\kke(m_{t_i})\big)\ d\tau
= f\big(x_i,m_{t_i},\kke(m_{t_i})\big) + g_i(t,x_i).
\end{align*}
By Proposition~\ref{prp:prev25} and the Jensen's inequality in the integral form, we have that
\begin{equation} \label{eq:gi}
\begin{split}
&\int_\rd \|g_i(t,\cdot\,)\|^2\ m_{t_i}(dx_i) \\
&\hspace{1.5cm}\leq \int_\rd \frac1{t-t_i} \int_{t_i}^t \|f\big(\traj{t_i}{x_i}{\tau},m_\tau,\kke(m_{t_i})\big) - f\big(x_i,m_{t_i},\kke(m_{t_i})\big)\|^2\ d\tau\ m_{t_i}(dx_i) \\
&\hspace{1.5cm}= \frac1{t-t_i} \int_{t_i}^t \int_\rd \| f\big(\traj{t_i}{x_i}{\tau},m_\tau,\kke(m_{t_i})\big) - f\big(x_i,m_{t_i},\kke(m_{t_i})\big)\|^2\ m_{t_i}(dx_i)\ d\tau \\
&\hspace{1.5cm}\leq C_3(R)^2 \cdot (t - t_i)^2
\leq C_3(R)^2 \cdot \delta_2^2,
\end{split}
\end{equation}
where constant $C_3$ depends only on the radius $R$ and the measure $\hm$.

By~\cite[Proposition~8.1.8]{Ambrosio_2005}
\[
m_t = \traj{t_i}{x_i}{t}\sharp m_{t_i}
\]
for all $t \in [t_i,t_{i+1}]$.
Notice that
\[
\traj{t_i}{x_i}{t}
= x_i + \int_{t_i}^t f\big(\traj{t_i}{x_i}{\tau},m_\tau,\kke(m_{t_i})\big)\ d\tau
= x_i + (t-t_i) \cdot f_i(t,x_i).
\]

Using Lemma~\ref{lmm:taylor}, we obtain
\begin{multline*}
\phivk(m_t) - \phivk(m_{t_i}) 
\leq \int_\rd \left( \frac{x_i-\yvk}{\vk^2}\right)^\tran \cdot (t-t_i) f_i(t,x_i)\ \big(\pivki\big)(dx_id\yvk) \\
+ \frac{(t-t_i)^2}{2\vk^2} \|f_i(t,\cdot\,)\|_{\ll_2(m_{t_i})}^2 + \bbrh.
\end{multline*}
Lemma~\ref{lmm:2Delta}, the H\"{o}lder's inequality, Lemma~\ref{lmm:w2.m.muvk} and inequality~\eqref{eq:gi} imply the following estimation of the first term:
\begin{align*}
\int_\rd &\left( \frac{x_i-\yvk}{\vk^2}\right)^\tran \cdot f_i(t,x_i)\ \big(\pivki\big)(dx_i d\yvk) \\
&\hspace{1cm}= \int_\rd \left( \frac{x_i-\yvk}{\vk^2}\right)^\tran \cdot f\big(x_i,m_{t_i},k(x_i)\big)\ \big(\pivki\big)(dx_i d\yvk) \\
&\hspace{3cm}+ \int_\rd \left( \frac{x_i-\yvk}{\vk^2}\right)^\tran \cdot g_i(t,x_i)\ \big(\pivki\big)(dx_i d\yvk) \\
&\hspace{1cm}\leq -2\dlt(r,R)
+ \bbr \cdot C_3(R) \cdot \delta_2,
\end{align*}
Furthermore, using the Jensen's inequality in the integral form, Remark~\ref{ass.ulin}, the triangle inequality and~\cite[Proposition~8.1.8]{Ambrosio_2005}, we estimate the second term:
\begin{align*}
\|f_i(t,\cdot\,)\|_{\ll_2(m_{t_i})}^2
&= \int_\rd \left\| \frac1{t-t_i} \int_{t_i}^t f\big(\traj{t_i}{x_i}{\tau},m_\tau,\kke(m_{t_i})\big)\ d\tau \right\|^2\ m_{t_i}(dx_i) \\
&\hspace{1cm}\leq \frac1{t-t_i} \int_{t_i}^t \int_\rd \| f\big(\traj{t_i}{x_i}{\tau},m_\tau,\kke(m_{t_i})\big) \|^2\ m_{t_i}(dx_i)\ d\tau \\
&\hspace{1.5cm}\leq \frac1{t-t_i} \int_{t_i}^t \int_\rd C_1^2 \big(1 + \|\traj{t_i}{x_i}{\tau}\| + \varsigma_2(m_\tau)\big)^2\ m_{t_i}(dx_i)\ d\tau \\
&\hspace{2.5cm}= \frac{C_1^2}{t-t_i} \int_{t_i}^t \big(1 + 4\varsigma_2(m_\tau) + 4\varsigma_2^2(m_\tau)\big)\ d\tau \\
&\hspace{3.5cm}\leq C_1^2 \big(1 + 4(\varsigma_2(\hm) + R) + 4(\varsigma_2(\hm) + R)^2\big).
\end{align*}
Thus, summing up these inequalities, we obtain
\begin{align*}
&\phivk(m_t) - \phivk(m_{t_i}) \\
&\hspace{1cm}\leq (t-t_i) \cdot \big( -2\dlt(r,R) + \bbr \cdot C_3(R) \cdot \delta_2 \big) \\
&\hspace{2cm} + \frac{t-t_i}{2\vk^2} \cdot \delta_2 \cdot C_1^2 \big(1 + 4(\varsigma_2(\hm) + R) + 4(\varsigma_2(\hm) + R)^2\big)
+ \bbrh \\
&\hspace{1cm}= \Bigg(
\bbr \cdot C_3(R)
+ \frac{C_1^2}{2\vk^2} \cdot \big(1 + 4(\varsigma_2(\hm) + R) + 4(\varsigma_2(\hm) + R)^2\big)
\Bigg) \cdot \delta_2 \cdot (t-t_i) \\
&\hspace{9.5cm}- 2\dlt(r,R) \cdot (t-t_i)
+ \bbrh.
\end{align*}
Finally, due to the conditions of the lemma, we arrive at the estimation
\[
\phivk(m_t) - \phivk(m_{t_i})
< -\dlt(r,R) \cdot (t-t_i) + \bbrh
< -\dlt(r,R) \cdot \delta_1 + \bbrh
\]
for all $t \in [t_i,\hatt]$.
Due to the fact that $\bbrh \to 0$, as $\varepsilon \to 0$, we can choose such $\varepsilon_1$, depending on $R$, $r$, $\delta_1$ and $\vk$, that
\[
\phivk(m_t) - \phivk(m_{t_i})
< 0
\]
for all $t \in [t_i,\hatt]$.
Taking into account continuity of the trajectory, we have that $\hatt = t_{i+1}$.
\end{proof}

Further, let us show that the trajectory generated by feedback $\kke$ attains the set $G_r$ on a finite time interval.

\begin{lmm} \label{lmm:smallest.N} 
In addition to conditions 1, 2 and 3 of Lemma~\ref{lmm:phivk.notinc}, assume $\varepsilon \in (0,\varepsilon_1)$.
Then, for each measure $m_* \in G_R^\vk$ there exists a number $N$ such that $\trajm{t_N}{\theta}{\kke} \in G_r$.
In particular, if $t_N$ is the smallest element of $\theta$ with that property, then it satisfies
\[
t_N
< \frac{\infout(R) + 2\bbrh}{2\dlt(r,R)}.
\]
\end{lmm}

\begin{proof}
We consider an arbitrary number $j$ such that, for all $i < j$, $m_{t_i} \not\in G_r$.
Sequentially application of Lemma~\ref{lmm:phivk.notinc} gives
\[
0 \leq \phivk(m_{t_j})
< \phivk(m_*) - \dlt(r,R) \cdot t_j + j \cdot \bbrh.
\]
Thus, due to the assumption that $m_* \in G_R^\vk$ and the definition of $\infout$,
\begin{equation} \label{eq:lmm.smallest.N}
t_j
< \frac{\phivk(m_*) + \bbrh}{\dlt(r,R)}
\leq \frac{\infout(R)}{2\dlt(r,R)} + \frac{j \cdot \bbrh}{\dlt(r,R)}.
\end{equation}
Collect all terms that depends on $j$ in the left-hand side:
\[
t_j - \frac{j \cdot \bbrh}{\dlt(r,R)} < \frac{\infout(R)}{2\dlt(r,R)}
\]
and apply the estimation $t_j \geq j \cdot \delta_1$.
Notice, that the choice of $\varepsilon$ gives us
\[
\delta_1 - \frac{\bbrh}{\dlt(r,R)} > 0.
\]
Choosing the greatest number $N$ such that time $t_N$ satisfies inequality~\eqref{eq:lmm.smallest.N}, we obtain $m_{t_N} \in G_r$.
\end{proof}
\medskip

Now, we will show that some open ball contains all trajectory points beyond some finite time.

\begin{lmm} \label{lmm:mt.in.ball}
In addition to the conditions of Lemma~\ref{lmm:smallest.N}, we assume that, for some $r>0$,
\begin{enumerate}
\item $\ds \modcon\left(\bbr\right) + \bbrh < \frac14 \infout(r)$,
\item $\ds \modcon(C_2(R) \cdot \delta) < \frac14 \infout(r)$.
\end{enumerate}
Then, $\theta$-trajectory $m_t = \trajm{t}{\theta}{\kke} \in \oall_r(\hm)$ for all $t \geq t_N$.
\end{lmm}

\begin{proof}
Consider a natural number $i \geq N$, where $N$ is given by Lemma~\ref{lmm:smallest.N}.
If $m_{t_i} \in G_r$, then, by Proposition~\ref{prp:prev24}, there exists $C_2(R)$ such that
\[
\phivk(m_t)
\leq \phi(m_t)
\leq \phi(m_{t_i}) + \modcon(C_2(R) \cdot \delta)
\leq \frac34 \infout(r),
\]
for all $t \in [t_i,t_{i+1}]$.
Hence,
\[
\phivk(m_{t_{i+1}})
\leq \frac34 \infout(r).
\]

If $m_{t_i} \not\in G_r$, then $m_{t_i} \not\in \ball_{\maxrad(r)}(\hm)$.
Applying Lemma~\ref{lmm:phivk.notinc}, we obtain $m_t \in G_r^\vk$ and
\[
\phivk(m_t)
\leq \frac12 \infout(r)
\leq \frac34 \infout(r)
\]
for all $t \in [t_i,t_{i+1}]$.
This argument can be repeated starting from time $t_{i+1}$, establishing:
\[
\phivk(m_t) \leq \frac34 \infout(r)
\]
for all $t \geq t_N$.

From Lemma~\ref{lmm:phi.lesser.phivk} and the first condition of this lemma, we derive
\[
\phi(m_t)
\leq \phivk(m_t) + \modcon\left(\bbr\right) + \bbrh
< \infout(r)
\]
Hence, $m_t \in \oall_r(\hm)$ for all $t \geq t_N$.
\end{proof}

Finally, we can formulate and prove the theorem about local stabilization.

\begin{thm} \label{thm:local} 
Let continuity equation~\eqref{eq:sys} admit a CLF.
Then, there exist functions
\begin{itemize}
\item $\vk_0\ :\ \ord^2 \to (0,1]$,
\item $T\ :\ \ord^2 \to \rrp$,
\item $\delta\ :\ \ord^2 \times (0,1] \to \rrpo$,
\end{itemize}
such that, for every $(r,R) \in \ord^2$, $\vk \in \big(0,\vk_0(r,R)\big)$ and $\hd \in \big(0,\delta(r,R,\vk)\big)$, one can find $\varepsilon_2 \in (0,1]$ satisfying the following properties:
\begin{enumerate}
\item $\ds m_t \in G_R^\vk$ for $t \geq 0$;
\item $\ds m_t \in \ball_r(\hm)$ for $t \geq T(r,R)$.
\end{enumerate}
Here $m_* \in G_R^\vk$, $m_t = \trajm{t}{\theta}{\kke}$, for every $\varepsilon \in (0,\varepsilon_2)$ and $\theta \in \trajset\big(\hd,\delta(r,R,\vk)\big)$.
\end{thm}

\begin{proof}
Choose some $\vk_0 \in (0,1]$ and $\varepsilon_2 \in \big(0, \min\{\varepsilon_0, \varepsilon_1\}\big)$ satisfying conditions~\eqref{eq:lmm.grvk.in.oall}, \eqref{eq:lmm.2Delta} and the first condition of Lemma~\ref{lmm:mt.in.ball}.
Define
\[
T(r,R) \triangleq \frac{\infout(R) + 2N \cdot \bbrha{\vk_0}{\varepsilon_2}}{2\dlt(r,R)}.
\]

For every $\vk \in (0,\vk_0)$, we construct a number $\delta'(r,R,\vk)$ as
\[
\delta'(r,R,\vk)
\triangleq \frac{\dlt(r,R)}{\Bigg( \bbra{\varepsilon_2}{\vk} \cdot C_3(R)
+ \frac{C_1^2}{2\vk^2} \cdot \big(1 + 4(\varsigma_2(\hm) + R) + 4(\varsigma_2(\hm) + R)^2\big) \Bigg)}.
\]
Every $\delta \in \big(0,\delta'(r,R,\vk)\big)$ consequently satisfies~\eqref{eq:lmm.phivk.notinc}.
Since $\phi \in \ucf$, there exists a number $\delta''(r,R)$ such that
\[
\modcon\big(C_2(R) \cdot \delta''(r,R)\big) < \frac14 \infout(r).
\]
Then, we select an arbitrary $\delta(r,R,\vk) < \min\big\{\delta'(r,R,\vk), \delta''(r,R)\big\}$.

Finally, we choose an arbitrary number $\hd \in \big(0,\delta(r,R,\vk)\big)$, partition $\theta \in \trajset\big(\hd,\delta(r,R,\vk)\big)$ and initial state $m_* \in G_R^\vk$.
Due to Lemma~\ref{lmm:phivk.notinc}, $m_t \in G_R^\vk$ for all $t \geq 0$.
Additionally, due to Lemma~\ref{lmm:mt.in.ball}, $m_t \in B_r(\hm)$ for all $t \geq T(r,R)$.
\end{proof}
\section{Global stabilization at the point} \label{sec:global}
In this section, we employ the results from the preceding section to establish the global stabilization (see Theorem~\ref{thm:global}).

\begin{dfn}
A feedback $k$ is called \dfnem{s-stabilizing} at the point $\hm$ for continuity equation~\eqref{eq:sys} if for every $(r,R) \in \ord^2$ there exist radius $M(R) > 0$, number $\delta(r,R) > 0$, and time $T(r,R) > 0$ such that, for all
\begin{itemize}
\item numbers $\hd \in \big(0,\delta(r,R)\big)$,
\item partitions $\theta \in \trajset\big(\hd,\delta(r,R)\big)$,
\item initial positions $m_*$ satisfying $W_p(m_*,\hm) \leq R$,
\end{itemize}
the following properties hold:
\begin{enumerate}
\item $W_2(m_t,\hm) \leq r$ for all $t \geq T(r,R)$;
\item $W_2(m_t,\hm) \leq M(R)$ for all $t \geq 0$;
\item $\lim\limits_{R \downarrow 0} M(R) = 0$;
\end{enumerate}
where $m_t = \trajm{t}{\theta}{k}$.
\end{dfn}

Below, in Theorem~\ref{thm:global}, we will show that if a continuity equation admits a CLP at the point $\hm$, then we can construct a feedback strategy $\hk$ (defined in~\eqref{eq:feedback.global}) that provides an s-stabilization at the point $\hm$.
The construction of this feedback relies on the auxiliary constructing presented below.

Let $(\phi,\psi)$ be a CLP.
We fix $Q_0 > 0$.
Since $\maxrad(R) \to \infty$ as $R \to \infty$ and $\maxrad(R) \to 0$ as $R \to 0$, we can construct a sequence $\{Q_i\}_{i=-\infty}^{+\infty}$ such that $Q_i \to +\infty$ as $i \to +\infty$, $Q_i \to 0$ as $i \to -\infty$ and
\begin{equation} \label{eq:R.dfn}
2 Q_i \leq \maxrad(Q_{i+1}).
\end{equation}
Define the sequence $\{q_i\}_{i=-\infty}^{+\infty}$ by:
\[
q_i \triangleq \frac12 \maxrad(Q_{i-1}).
\]
For each integer number $i$, we have that 
\[
q_{i+1} < \maxrad(Q_i) \leq Q_i \leq 2 Q_i \leq \maxrad(Q_{i+1}).
\]
Additionally, by the definition of the function $\maxrad$ we obtain the inclusion
\[
\ball_{q_i}(\hm) \subset G_{Q_{i-1}}^{\vk_{i-1}}.
\]

Fix an arbitrary integer number $i$.
Theorem~\ref{thm:local} provides the existence of numbers $\vk_i \in (0,1]$, $T_i > 0$ and $\delta_i > 0$ with the following properties:
for every $\hd_i \in (0,\delta_i)$, one can find $\varepsilon_i(q_i,Q_i,\delta_i) \in (0,1]$ such that $G_{Q_i}^{\vk_i}$ remains invariant w.r.t. the $\theta$-trajectory $m_\cdot = \trajm{\cdot}{\theta}{\kkea{\vk_i}{\varepsilon}}$, for all $\varepsilon \in (0,\varepsilon_i)$ and $\theta \in \trajset(\hd_i,\delta_i)$.
Furthermore, $m_t \in \ball_{q_i}(\hm)$ for all $t \geq T_i$.

Finally, we introduce the set
\[
H_i \triangleq G_{Q_{i+1}}^{\vk_{i+1}} \setminus G_{Q_i}^{\vk_i}.
\]
Below we will assume that the feedback $\kkea{\vk_i}{\varepsilon_i(q_i,Q_i,\delta_i)}$ works on the set $H_i$.
\medskip

Let us prove an important property of the sets $\{H_i\}_{i=-\infty}^{+\infty}$.

\begin{lmm} 
\[
\bigcup\limits_{i=-\infty}^{+\infty} H_i \cup \{\hm\}
= \pp_2(\rd).
\]
\end{lmm}

\begin{proof}
By the definitions of $Q_i$ and $\maxrad(Q_i)$, the following chain of inclusions holds for any integer number $i$:
\[
G_{Q_i}^{\vk_i} \subseteq \ball_{Q_i}(\hm) \subset \ball_{2 Q_i}(\hm) \subseteq G_{Q_{i+1}}^{\vk_{i+1}}.
\]
Consequently, $H_i \subseteq G_{Q_{j+1}}^{\vk_{j+1}}$ for all $i < j$.

For every integer numbers $M < N$, we have the inclusion:
\begin{equation} \label{eq:diff.g.h}
\ball_{Q_N}(\hm) \setminus \ball_{Q_M}(\hm)
\subseteq G_{Q_{N+1}}^{\vk_{N+1}} \setminus G_{Q_M}^{\vk_M}
= \bigcup\limits_{i=M+1}^{N} H_i.
\end{equation}
Taking limits as $M \to -\infty$ and $N \to +\infty$ in the constructed sequence $\{Q_i\}_{i=-\infty}^{+\infty}$ yields the lemma's conclusion.
\end{proof}

Using this property, we construct the global feedback $\hk: \pp_2(\rd) \to U$ by the following rule:
\begin{equation} \label{eq:feedback.global}
\hk(m) \triangleq \begin{cases}
\kkea{\vk_i}{\varepsilon_i(q_i,Q_i,\delta_i)}(m), &m \in H_{i-1}; \\
u_o, &m = \hm,
\end{cases}
\end{equation}
where $u_o \in U$ is an arbitrary control.
\medskip

Now, we will show that the set $G_R^\vk$ is also remains invariant under the feedback $\hk$.

\begin{lmm} \label{lmm:sampling.feedback}
Let
\begin{enumerate}
\item $\hd_i \in \left(0,\min\left\{ \delta_i , \frac{Q_{i-1}}{C_2(Q_i)} \right\}\right)$,
\item $\theta \in \trajset\left(\hd_i,\ \min\left\{ \delta_i , \frac{Q_{i-1}}{C_2(Q_i)} \right\}\right)$.
\end{enumerate}
Then, for any initial condition $m_* \in G_{R_i}^{\vk_i}$, the entire $\theta$-trajectory $m_\cdot = \trajm{\cdot}{\theta}{\hk}$ remains within $G_{Q_i}^{\vk_i}$.
\end{lmm}

\begin{proof}
Consider a natural number $j$.
Let us proof that $m_{t_j} \in G_{Q_i}^{\vk_i}$ imply $m_t \in G_{Q_i}^{\vk_i}$ for all $t \in [t_j,t_{j+1}]$.
From equality~\eqref{eq:diff.g.h}, we have the decomposition:
\[
G_{Q_i}^{\vk_i}
= \bigcup\limits_{l=-\infty}^{i-1} H_l \cup \{ \hm \}.
\]
We will analyze three cases.

\textbf{Case 1:} $\ds m_{t_j} \in H_{i-1}$.
The result follows directly from Theorem~\ref{thm:local}.

\textbf{Case 2:} $\ds m_{t_j} \in H_{l-1}$ for some $l \leq i-1$.
Let $\hatt$ be the maximal time $T \in [t_j,t_{j+1}]$ satisfying $m_t \in \ball_{Q_i}(\hm)$ for all $t \in [t_j,T]$.
Using the triangle inequality, Proposition~\ref{prp:prev24} and the second assumption of the lemma, we obtain the inequality
\[
W_2(m_t,\hm)
\leq W_2(m_t,m_{t_j}) + W_2(m_{t_j},\hm)
\leq C_2(Q_i) (t - t_j) + Q_l
< Q_{i-1} + Q_l
\leq 2 Q_{i-1},
\]
for all $t \in [t_i,\hatt]$.
The trajectory's continuity and condition~\eqref{eq:R.dfn} imply $\hatt = t_{j+1}$, hence
\[
m_t \in \ball_{2Q_{i-1}} \subseteq G_{Q_i}^{\vk_i}, \quad \text{for } t \in [t_j,t_{j+1}].
\]

\textbf{Case 3:} $\ds m_{t_j} = \hm$.
Using the argument of previous case and replacing $Q_l$ in inequality with $0$, we obtain
\[
m_t \in \ball_{Q_{i-1}} \subseteq G_{Q_i}^{\vk_i}, \quad \text{for } t \in [t_j,t_{j+1}].
\]
\end{proof}

Now, finally, we can prove the theorem about s-stabilization.

\begin{thm} \label{thm:global} 
If continuity equation~\eqref{eq:sys} admits CLP, then the feedback $\hk$ is s-stabilizing.
\end{thm}

\begin{proof}
Define $N(R)$ as the smallest integer satisfying
\[
\ball_R(\hm) \subseteq G_{Q_{N(R)}}^{\vk_{N(R)}}.
\]
Then, we choose an arbitrary number $K(r) < N(R)$ such that
\[
G_{Q_{K(r)}}^{\vk_{K(r)}} \subseteq \ball_r(\hm).
\]
Let us put
\begin{align*}
\delta(r,R) &\triangleq \min\left\{ \min\limits_{K(r) \leq i \leq N(R)} \delta_i ,\ \frac{Q_{K(r)-1}}{C_2(Q_{N(R)})} \right\},\\
0 < \hd &< \delta(r,R),\\
T(r,R) &\triangleq \sum\limits_{i=K(r)+1}^{N(R)} \big(T_i+\delta(r,R)\big),\\
M(R) &\triangleq Q_{N(R)}.
\end{align*}

We consider a partition $\theta \in \trajset\big(\hd,\delta(r,R)\big)$ and $\theta$-trajectory $m_\cdot = \trajm{\cdot}{\theta}{\kkea{\vk}{\varepsilon(r,R)}}$ starting from some measure $m_* \in \ball_R(\hm)$.
Lemma~\ref{lmm:sampling.feedback} guarantees that, for all $t \geq 0$,
\[
m_t \in G_{Q_{N(R)}}^{\vk_{N(R)}} \subseteq \ball_{Q_{N(R)}}(\hm) = \ball_{M(R)}(\hm).
\]

If $m_* \not\in G_{Q_{K(r)}}^{\vk_{K(r)}}$, then there exists an integer number $N' \in \big[K(r)+1,N(R)\big]$ with $m_* \in H_{N'-1}$.
Theorem~\ref{thm:local} ensures that
\[
m_{T_{N'}} \in \ball_{q_{N'}}(\hm) \subset G_{Q_{N'-1}}^{\vk_{N'-1}}.
\]
Thus, by the trajectory's continuity, there exists an integer number $j$ with $t_j < T_{N'}+\delta$ such that $m_t \in G_{Q_{N'-1}}^{\vk_{N'-1}}$ for some $t \in (t_{j-1},t_j]$.
Applying this argument sequentially, we obtain that
\[
m_{t_i} \in G_{Q_{K(r)}}^{\vk_{K(r)}} \subseteq \ball_r(\hm)
\]
for some partition time $t_i < T(r,R)$.
Additionally, Lemma~\ref{lmm:sampling.feedback} guarantees $m_t \in G_{Q_{K(r)}}^{\vk_{K(r)}} \subseteq \ball_r(\hm)$ for all $t \geq t_i$ and, consequently, for all $t \geq T(r,R)$.

Due to the choice of $N(R)$, we have that
\[
\lim\limits_{R \to 0} N(R) = -\infty,
\]
and, consequently,
\[
\lim\limits_{R \to 0} M(R)
= \lim\limits_{R \to 0} Q_{N(R)}
= 0,
\]
which completes the proof.
\end{proof}
\printbibliography
\end{document}